\documentclass[reqno,11pt]{amsart}
\usepackage{amssymb,latexsym}
\usepackage{array,multirow,makecell}
\usepackage{mathtools}
\usepackage{IEEEtrantools}
\usepackage{amsmath}
\DeclarePairedDelimiter{\abs}{\lvert}{\rvert}
\setcellgapes{1pt}
\makegapedcells
\newtheorem{theorem}{Theorem}[section]
\newtheorem{corollary}[theorem]{Corollary}

\newtheorem{proposition}[theorem]{Proposition}
\newtheorem{korselt's criterion}[theorem]{Korselt's criterion}
\newtheorem{definition}[theorem]{Definition}

\newtheorem{example}[theorem]{Example}
\newtheorem{examples}[theorem]{Examples}

\numberwithin{equation}{section}
\begin{document}
\title[Korselt rational bases of  prime powers ]
{Korselt rational bases of  prime powers}
\author{Nejib Ghanmi}
\address[Ghanmi]{(1) Preparatory Institut of Engineering Studies of Tunis, Montfleury, Tunisia.}

 \address[]{\hspace{1.65cm}(2) University College of Jammum, Department of Mathematics, Mekkah, Saudi Arabia.}
\email{neghanmi@yahoo.fr\; and \; naghanmi@uqu.edu.sa}

\thanks{}



\begin{abstract}

Let $N$ be a positive integer,  $\mathbb{A}$ be a subset of $\mathbb{Q}$ and $\alpha=\dfrac{\alpha_{1}}{\alpha_{2}}\in \mathbb{A}\setminus \{0,N\}$. $N$ is called an \emph{$\alpha$-Korselt number} (equivalently $\alpha$ is said an \emph{$N$-Korselt base})
 if  $\alpha_{2}p-\alpha_{1}$ divides $\alpha_{2}N-\alpha_{1}$ for every prime divisor $p$ of $N$. By the \emph{Korselt set} of  $N$ over $\mathbb{A}$, we
mean the set $\mathbb{A}$-$\mathcal{KS}(N)$ of all $\alpha\in \mathbb{A}\setminus \{0,N\}$ such that $N$ is an $\alpha$-Korselt number.

In this paper we determine explicitly for a given prime number $q$  and an integer $l\in \mathbb{N}\setminus \{0,1\}$, the set $\mathbb{Q}$-$\mathcal{KS}(q^{l})$ and we establish some connections between the $q^{l}$-Korselt bases  in $\mathbb{Q}$ and others in $\mathbb{Z}$. The case of
$\mathbb{A}=\mathbb{Q}\cap[-1,1[$ is studied where we prove that $(\mathbb{Q}\cap[-1,1[)$-$\mathcal{KS}(q^{l})$ is empty if and only if $l=2$.

Moreover,  we show that  each nonzero rational $\alpha$ is an $N$-Korselt base for infinitely many numbers $N=q^{l}$ where $q$ is a prime number and $l\in\mathbb{N}$.

\end{abstract}

\maketitle

\section{Introduction}

By extending the Korselt criterion, Bouallegue-Echi-Pinch~\cite{BouEchPin,echi} stated the notion of Korselt numbers and  gave rise to a new kind of generalized Carmichael numbers. As known, Carmichael numbers ~\cite{Bee,Car2} are squarefree composite  numbers $N$ verifying $a^{N-1}\equiv 1 \pmod  N$  for each integer $a$ with
$\gcd(a,N) = 1$ . These numbers are characterized by Korselt as follows:

\begin{quotation}
\textbf{Korselt's criterion (1899)~\cite{Kor}}: A composite integer  $N>1$ is carmichael if and only if $N$ is squarefree and  $p-1$ divides $N-1$ for all prime factors  $p$ of $N$.
\end{quotation}
\medskip

In 1910, Carmichael ~\cite{Car1} found  $561=3.11.17$ as  the first and smallest  number  verifying the Korselt criterion, this explains the name "Carmichael numbers".
Although, Korselt was the first who state the basic properties of Carmichael numbers, he could not find any example; if he had just done a few computation, then surely
``Carmichael numbers" would have been known as ``Korselt numbers".

In honor to Korselt, Bouallegue-Echi-Pinch ~\cite{BouEchPin,echi} introduced the concept of $\alpha$-Korselt numbers with $\alpha\in\mathbb{Z}$, where the Carmichael numbers are exactly the squarefree composite  1-Korselt numbres.

Ghanmi ~\cite{Ghanmi2} introduced the  notion of $\mathbb{Q}$-Korselt numbers as extension  of Korselt numbers to $\mathbb{Q}$ by setting the following definitions:

\begin{definition}\label{def1}\rm Let $N\in \mathbb{N}\setminus\{0,1\}$,   $\alpha=\dfrac{\alpha_{1}}{\alpha_{2}}\in \mathbb{Q}\setminus \{0\}$ with $\mathbb{A}$ a subset of $\mathbb{Q}$.
Then
\begin{enumerate}
  \item $N$ is said to be an \emph{$\alpha$-Korselt number\index{Korselt number}} (\emph{$K_{\alpha}$-number}, for
short), if $N\neq \alpha$ and $\alpha_{2}p-\alpha_{1}$ divides $\alpha_{2}N-\alpha_{1}$ for
every prime divisor $p$ of $N$.

\item By the \emph{$\mathbb{A}$-Korselt set}\index{Korselt set} of the number $N$ (or the Korselt set of $N$ over \emph{$\mathbb{A}$}) , we mean the set $\mathbb{A}$-$\mathcal{KS}(N)$ of
all $\beta\in \mathbb{A}\setminus\{0,N\}$ such that $N$ is a $K_{\beta}$-number.
  \item The cardinality of $\mathbb{A}$-$\mathcal{KS}(N)$ will be called the \emph{$\mathbb{A}$-Korselt
weight}\index{Korselt weight} of $N$ (or the Korselt weight of $N$ over \emph{$\mathbb{A}$}) ; we denote it by $\mathbb{A}$-$\mathcal{KW}(N)$.

\end{enumerate}

\end{definition}

In this paper we will concentrate on the study of the rational numbers $\alpha\in \mathbb{Q}\setminus \{0\}$ for which a given number $N$ is an $\alpha$-Korselt number, for this we state the following definition.
\begin{definition}\label{def2}\rm Let $N\in \mathbb{N}\setminus\{0,1\}$, $\alpha\in \mathbb{Q}\setminus\{0\}$ and $\mathbb{B}$ be a subset of $\mathbb{N}$. Then

\begin{enumerate}
\item  $\alpha$ is called \emph{$N$-Korselt base\index{Korselt base}} (\emph{$K_{N}$-base}, for
short) if $N$ is a \emph{$K_{\alpha}$-number}.
\item By the \emph{$\mathbb{B}$-Korselt set}\index{Korselt base set} of the base $\alpha$ (or the Korselt set of the base $\alpha$ over \emph{$\mathbb{B}$}), we mean the set $\mathbb{B}$-$\mathcal{KS}(B(\alpha))$ of
all $M\in \mathbb{B}$ such that $\alpha$ is a
$K_{M}$-base.
\item The cardinality of  $\mathbb{B}$-$\mathcal{KS}(B(\alpha))$ will be called the \emph{$\mathbb{B}$-Korselt
weight}\index{Korselt  weight} of the base $\alpha$ (or the Korselt weight of the base $\alpha$ over \emph{$\mathbb{B}$}); we denote it by $\mathbb{B}$-$\mathcal{KW}(B(\alpha))$.
\end{enumerate}
\end{definition}

For more convenience, the set $\bigcup\limits_{\alpha \in\mathbb{Q}} ( \mathbb{N} \text{-}\mathcal{KS}(B(\alpha)))$ is  simply called the set of $\mathbb{Q}$-Korselt numbers or the set of rational Korselt numbers or the set of  Korselt numbers over $\mathbb{Q}$. Also, the set $\bigcup\limits_{N \in\mathbb{N}} ( \mathbb{Q} \text{-}\mathcal{KS}(N))$ is called the set of Korselt rational bases or the set of $\mathbb{N}$-Korselt bases in $\mathbb{Q}$ or  the set of  Korselt rational bases  over $\mathbb{N}$.

Since their creation, Korselt numbers are well investigated specially in~\cite{rassasi,Ghanmi,BouEchPin,Ghanmi2,echi}, where the majority of cases except in~\cite{rassasi}, the authors restrict their studies on the case of squarefree composite numbers. In this paper we focus our study on the prime power Korselt numbers, where we manage to determine explicitly the Korselt set and the Korselt weight  of a prime power number over  $\mathbb{Q}$. Therefore many properties of prime power Korselt number over  $\mathbb{Z}$ immediately follows.

We organize this work as follows. In section $1$ we recall some useful properties of the $q^{l}$-Korselt rational bases. The  $q^{l}$-Korselt set over the sets $ \mathbb{Q}$ and $\mathbb{Z}$ are determined  explicitly in section $2$. Hence, we simply deduce each of the weights $\mathbb{Q}$-$\mathcal{KW}(q^{l})$ and $\mathbb{Z}$-$\mathcal{KW}(q^{l})$. Moreover, we establish some relations between the  $q^{l}$-Korselt bases in $\mathbb{Z}$  and others in $\mathbb{Q}$. In section  $3$, we prove that  each nonzero rational $\alpha$ is a $K_{N}$-base for infinitely many prime powers $N$. Finally, the section $4$ is devoted to study the $q^{l}$-Korselt rational bases in $[-1,1[$, where we present a necessary and sufficient condition for a given $\alpha\in\mathbb{Q}\cap[-1,1[$ to be a $q^{l}$-Korselt rational base. This leads to conclude that for each prime number $q$, the
$q^{2}$-Korselt set over $\mathbb{Q}\cap[-1,1[$ is empty.

Throughout this paper we assume that $l\geq 2$ and for $\alpha=\dfrac{\alpha_{1}}{\alpha_{2}}\in\mathbb{Q}$, we will suppose without loss of generality that $\alpha_{2}>0$, $\alpha_{1}\in \mathbb{Z}$ and $\gcd(\alpha_{1},\alpha_{2})=1$.

\section{Properties of  $q^{l}$-Korselt rational bases}

\begin{proposition}\label{carac0} Let $q$ be a prime number, $\alpha=\dfrac{\alpha_{1}}{\alpha_{2}}\in\mathbb{Q}\setminus\{0\}$ and

$l\in\mathbb{N}\setminus\{0\}$.

\begin{enumerate}
  \item If $\gcd(\alpha_{1},q)=1$, then
  $$\alpha\in\mathbb{Q}\text{-}\mathcal{KS}(q^{l}) \ \ \text{if and only if} \ \ (\alpha_{2}q-\alpha_{1} )\mid(q^{l-1}-1).$$
  \item If $\alpha_{1}=\alpha_{1}^{'}q$, then
  $$\alpha\in\mathbb{Q}\text{-}\mathcal{KS}(q^{l})\ \ \text{if and only if} \ \ (\alpha_{2}-\alpha_{1}^{'})\mid( q^{l-1}-1).$$
\end{enumerate}

\end{proposition}

\begin{proof}
We have $\alpha\in\mathbb{Q}$-$\mathcal{KS}(q^{l})$ if and only if $ (\alpha_{2}q-\alpha_{1})\mid(\alpha_{2}q^{l}-\alpha_{1})=\alpha_{2}q(q^{l-1}-1)+\alpha_{2}q-\alpha_{1}.$
As $\gcd(\alpha_{2}q-\alpha_{1},\alpha_{2})=1$, it follows that

\begin{equation}\label{eq1}
\alpha\in\mathbb{Q}\text{-}\mathcal{KS}(q^{l}) \Leftrightarrow \alpha_{2}q-\alpha_{1}\mid q(q^{l-1}-1)
  \end{equation}

Hence, if $\gcd(\alpha_{1},q)=1$ then  $\eqref{eq1}$ gives

$$\alpha\in\mathbb{Q}\text{-}\mathcal{KS}(q^{l}) \Leftrightarrow (\alpha_{2}q-\alpha_{1})\mid (q^{l-1}-1),$$

and if $\alpha_{1}=\alpha_{1}^{'}q$, then by $\eqref{eq1}$ we get

$$\alpha\in\mathbb{Q}\text{-}\mathcal{KS}(q^{l}) \Leftrightarrow (\alpha_{2}-\alpha_{1}^{'})\mid (q^{l-1}-1).$$

\end{proof}
The following result is an immediate consequence of Definition ~\ref{def1}.
\begin{corollary}\label{firstprop}

Let $q$ be a prime number,  $l\in\mathbb{N}\setminus\{0\}$ and $\dfrac{\alpha_{1}^{'}q}{\alpha_{2}}\in \mathbb{Q}\setminus\{0\}$ with $\gcd(\alpha_{1}^{'}q,\alpha_{2})=1$. Then,
$\dfrac{\alpha_{1}^{'}q}{\alpha_{2}}\in\mathbb{Q}$-$\mathcal{KS}(q^{l})$ if and only if \, $\dfrac{\alpha_{2}q}{\alpha_{1}^{'}}\in\mathbb{Q}$-$\mathcal{KS}(q^{l})$.
\end{corollary}

By the next two results, we give information about  $\mathbb{Q}$-$\mathcal{KS}(q^{l})$.
\begin{proposition}\label{bound0}Let $\alpha\in\mathbb{Q}$-$\mathcal{KS}(q^{l})$. If  $\alpha=\alpha^{'}q$, then   $$2-q^{l-1}\leq\alpha^{'}\leq \dfrac{q^{l-1}+1}{2}.$$

\end{proposition}
\begin{proof}
Let $N=q^{l}$ and $\alpha=\dfrac{\alpha_{1}}{\alpha_{2}}=\alpha^{'}q \in\mathbb{Q}$-$\mathcal{KS}(q^{l})$ with $\gcd(\alpha_{2},\alpha_{1})=1.$

 \begin{description}
   \item[Case 1] Assume that $\alpha^{'} >0$. First, we show by contradiction that $\alpha^{'} < q^{l-1}$. If $\alpha^{'} \geq q^{l-1}$, then $\alpha_{1}-\alpha_{2}N=\alpha_{2}q(\alpha^{'} -q^{l-1})\geq0$,  hence $\alpha_{1}-\alpha_{2}q>\alpha_{1}-\alpha_{2}N\geq0.$ But as   $\alpha_{1}-\alpha_{2}q$ divides $\alpha_{1}-\alpha_{2}N$, it follows that  $\alpha_{1}-\alpha_{2}N=0$, and so $\alpha =N$ which is impossible. Thus, $\alpha^{'} < q^{l-1}$.

Now, let us show that $\alpha^{'}\leq\dfrac{q^{l-1}+1}{2}.$  We can suppose that $\alpha^{'}>1$ (otherwise  $\alpha^{'} \leq 1$ and so the result is immediate).

Since $\alpha_{2}N-\alpha_{1} =k(\alpha_{2}q-\alpha_{1})\ \ \text{with} \ \ k \in \mathbb{Z}\setminus \{0\}$, it follows by dividing both sides by $\alpha_{2}q$, that $q^{l-1}-\alpha^{'} =k(1-\alpha^{'})$. As $\alpha^{'} > 1$, it yields that $\alpha^{'}-1=|1-\alpha^{'}|\leq |q^{l-1}-\alpha^{'}|=q^{l-1}-\alpha^{'}$, and so $\alpha^{'}\leq\dfrac{q^{l-1}+1}{2}$.
   \item[Case 2] Now, suppose that $\alpha^{'} <0 $ (i.e.  $\alpha_1 <0 $).

   Since $\alpha_{2}N-\alpha_{1} =k(\alpha_{2}q-\alpha_{1})$ with $k \in
\mathbb{N}$, we  claim that $k \geq 2$, otherwise $l = 1$ which is impossible. So, $N-\alpha=k(q-\alpha)\geq 2(q-\alpha)$, which implies that  $ 2q-N\leq\alpha$. Thus, $2-q^{l-1}\leq\alpha^{'}$.
 \end{description}

\end{proof}

\begin{proposition}\label{bound1}  Assume that $\alpha=\dfrac{\alpha_{1}}{\alpha_{2}}\in\mathbb{Q}$-$\mathcal{KS}(q^{l})$ with  $\gcd(\alpha_{1},q)=1$. Then    $$1+q-q^{l-1}\leq\alpha\leq q^{l-1}+q-1.$$
\end{proposition}

\begin{proof}

Let $N=q^{l}$ and $\alpha=\dfrac{\alpha_{1}}{\alpha_{2}}\in\mathbb{Q}$-$\mathcal{KS}(N)$. Then, by Proposition ~\ref{carac0}$(1)$ we have
    $\alpha_{2}q-\alpha_{1}\mid q^{l-1}-1$,
hence   $1-q^{l-1}\leq \alpha_{2}q-\alpha_{1}\leq  q^{l-1}-1$. Since $\alpha_{2}\geq1$,  this implies that
$$1-q^{l-1}\leq\dfrac{1-q^{l-1}}{\alpha_{2}}\leq q-\alpha\leq  \dfrac{q^{l-1}-1}{\alpha_{2}}\leq q^{l-1}-1.$$

Thus,  $$1+q-q^{l-1}\leq\alpha\leq q^{l-1}+q-1.$$

\end{proof}

The  following two examples  show that in Propositions ~\ref{bound0} and ~\ref{bound1} the upper and lower bounds are attained (except in Proposition ~\ref{bound0} where   only the upper bound is attained for $N=2^{2}$).
\begin{examples}\label{Attein1}\rm
\begin{enumerate}
  \item    Let $\alpha=\dfrac{q^{l}+q}{2}$. Then $q^{l}-\alpha=\dfrac{q^{l}-q}{2}=-(q-\alpha)$, therefore $\alpha$ is a $q^{l}$-Korselt base.

  Now, by taking $\beta=2q-q^{l}$, we get  $q^{l}-\beta=2(q^{l}-q)=2(q-\beta)$. Thus $\beta$ is a $q^{l}$-Korselt base.

  \item  Let $\alpha=q^{l-1}+q-1$. Then $q^{l}-\alpha=(1-q)(1-q^{l-1})=(1-q)(q-\alpha)$, hence $\alpha$ is a $q^{l}$-Korselt base.

   If  $\beta=1+q-q^{l-1}$, then $q^{l}-\beta=(q+1)(q^{l-1}-1)=(q+1)(q-\beta)$. Hence $\beta$ is a $q^{l}$-Korselt base.

\end{enumerate}

\end{examples}

\bigskip

The following result gives relations between the Korselt sets of two powers of the same prime number as an immediate consequence of Definition~\ref{def1}.

\begin{proposition}\label{inclusion1}Let  $l,k\in\mathbb{N}\setminus\{0,1\}$, $q$ be a prime number and $\mathbb{A}$ be a subset of $\mathbb{Q}$.  Then
\begin{enumerate}
  \item  $\mathbb{A}\text{-}\mathcal{KS}(q^{2})\subseteq  \mathbb{A}\text{-}\mathcal{KS}(q^{l})$.
  \item In general , if $(k-1)\mid(l-1)$ then $\mathbb{A}\text{-}\mathcal{KS}(q^{k})\subseteq  \mathbb{A}\text{-}\mathcal{KS}(q^{l})$.

 \end{enumerate}

 \end{proposition}

\bigskip

  This result can be generalized as follows.
 \begin{proposition}\label{inclusion2}

  Let  $l,k\in\mathbb{N}\setminus\{0,1\}$, $q$ be a prime number and $\mathbb{A}$ be a subset of $\mathbb{Q}$.  Then the following hold.
  \begin{enumerate}
  \item If $\gcd(l-1,k-1)=m-1$, then  $$\mathbb{A}\text{-}\mathcal{KS}(q^{l})\cap  \mathbb{A}\text{-}\mathcal{KS}(q^{k})=\mathbb{A}\text{-}\mathcal{KS}(q^{m})$$

    \item If $1\in \mathbb{A}$ then $\bigcap\limits_{\substack{t\in\mathbb{N}\setminus\{0,1\}\\
                                                               q\,\text{prime}}}
    (\mathbb{A}\text{-}\mathcal{KS}(q^{t}))=\{1\}$.
\end{enumerate}
\end{proposition}

\begin{proof}
\begin{enumerate}
\item
Since $m-1$ divides each of \, $l-1$ and $k-1$, then by Proposition ~\ref{inclusion1}, we get

$$\mathbb{A}\text{-}\mathcal{KS}(q^{m})\subseteq \mathbb{A}\text{-}\mathcal{KS}(q^{l})\cap  \mathbb{A}\text{-}\mathcal{KS}(q^{k}).$$

On the other hand, if  $\alpha=\dfrac{\alpha_{1}}{\alpha_{2}}\in\mathbb{A}\text{-}\mathcal{KS}(q^{l})\cap  \mathbb{A}\text{-}\mathcal{KS}(q^{k})$ then

$$\left\{\begin{array}{rrr}
\alpha_{2}q-\alpha_{1} & \mid & q^{l}-q=q( q^{l-1}-1)\\
\alpha_{2}q-\alpha_{1}  & \mid & q^{k}-q=q( q^{k-1}-1)\\
\end{array}
\right. $$
hence $$\alpha_{2}q-\alpha_{1} \mid  q \gcd(q^{l-1}-1, q^{k-1}-1).$$

But as $$\gcd(q^{l-1}-1, q^{k-1}-1)=q^{\gcd(l-1,k-1)}-1=q^{m-1}-1,$$ it follows that
$(\alpha_{2}q-\alpha_{1})\mid  q(q^{m-1}-1).$
Thus $\alpha\in\mathbb{A}\text{-}\mathcal{KS}(q^{m}).$

So, we conclude that   $\mathbb{A}\text{-}\mathcal{KS}(q^{l})\cap  \mathbb{A}\text{-}\mathcal{KS}(q^{k})=\mathbb{A}\text{-}\mathcal{KS}(q^{m})$.

\item By Proposition ~\ref{inclusion2}$(1)$ we have $$\bigcap\limits_{{t\in\mathbb{N}\setminus\{0,1\}}}\left(\mathbb{A}\text{-}\mathcal{KS}(q^{t})\right)=\mathbb{A}\text{-}\mathcal{KS}(q^{2}),$$

    hence

    $$ \bigcap\limits_{\substack{t\in\mathbb{N}\setminus\{0,1\}\\
                                  q\,\text{prime}}}
     \left(\mathbb{A}\text{-}\mathcal{KS}(q^{t})\right)=\bigcap\limits_{{q \, \text{prime}}}\left(\mathbb{A}\text{-}\mathcal{KS}(q^{2})\right).$$

\bigskip

Let $\alpha\in\bigcap\limits_{{q \, \text{prime}}}(\mathbb{A}\text{-}\mathcal{KS}(q^{2}))$ . We claim that if $(q^{2}-�\alpha)/(q-�\alpha)\in \mathbb{Z}$ for all
primes $q$ and  $\alpha\neq 0$ then $\alpha=1$. Well,

$$\dfrac{q^{2}-�\alpha}{q-�\alpha}= q+(\alpha+\frac{\alpha^{2}-\alpha}{q-\alpha}),$$
and taking $q$ sufficiently large so that $\dfrac{\alpha^{2}-\alpha}{q-\alpha}$ is smaller than
the reciprocal of the denominator of $\alpha$ we get a contradiction unless $\alpha^{2}=\alpha$, and with $\alpha\neq 0$ that forces $\alpha=1$.

Finally, we conclude that   $$\bigcap\limits_{\underset{t\in\mathbb{N}\setminus\{0,1\}}{q \, \text{prime}}}\left(\mathbb{A}\text{-}\mathcal{KS}(q^{t})\right)=\{1\}.$$
\end{enumerate}
\end{proof}

\section{Connections between  Korselt rational bases of $q^{l}$.}

In this section we will determine explicitly for a given prime number $q$  and an integer $l\in \mathbb{N}\setminus \{0,1\}$, the set $\mathbb{Q}$-$\mathcal{KS}(q^{l})$. Further, we state some connections between $q^{l}$-Korselt bases  in $\mathbb{Q}$ and others in $\mathbb{Z}$.

First, we determine by the next two results,  the Korselt sets of a prime power over $\mathbb{Z}$ and $\mathbb{Q}$.
\begin{theorem}\label{carac3}

Let $N=q^{l}$ be such that $l\geq2$. Then

$$\mathbb{Q}\text{-}\mathcal{KS}(N)=\bigcup\limits_{{s\in\mathbb{Z}\setminus\{0\}}} \{ q+\dfrac{d}{s}; \ \ \text{where} \ \  d \in \mathbb{N}\ \  \text{and} \ \  d\mid (q^{l}-q) \}\setminus\{0,N\}.$$

  \end{theorem}

  \begin{proof}
Suppose that $\alpha=\dfrac{\alpha_{1}}{\alpha_{2}}\in\mathbb{Q}$-$\mathcal{KS}(N)$. Then, two cases are to be considred
\begin{itemize}
  \item If $\alpha_{1}=\alpha_{1}^{'}q$, then   $d=\alpha_{1}^{'}-\alpha_{2}\mid (q^{l-1}-1)$ by Proposition ~\ref{carac0}$(2)$. Hence, $\alpha=q(1+\frac{d}{\alpha_{2}})$.

  \item If $\gcd(q,\alpha_{1})=1$, then as $d=\alpha_{1}-\alpha_{2}q\mid(q^{l-1}-1)$ by Proposition ~\ref{carac0}$(1)$, we get  $\alpha=q+\frac{d}{\alpha_{2}}$.
\end{itemize}

The converse, in the two cases, is immediate.
\end{proof}

 \bigskip

Now, immediately by Theorem~\ref{carac3}, we get the following result.
\begin{theorem}\label{carac4}

Let $N=q^{l}$ be such that $l\geq2$. Then

   $$\mathbb{Z}\text{-}\mathcal{KS}(N)= \{ q\pm d \,; \ \ \text{where} \ \  d \in \mathbb{N}\ \  \text{and}\ \ d\mid (q^{l}-q) \}\setminus\{0,N\}.$$

\end{theorem}

\bigskip

  For  a prime number $q$ and $l\in\mathbb{N}\setminus\{0,1\}$,  it's clear  by Theorem ~\ref{carac4} that  the set $\mathbb{Z}\text{-}\mathcal{KS}(q^{l}) $ is finite and the
    $\mathbb{Z}$-Korselt weight of $q^{l}$ is equal to the cardinality of the divisors  in $\mathbb{Z}$ of $q^{l}-q$ minus two.
   Hence, $\mathbb{Z}\text{-}\mathcal{KW}(q^{l})=4\sigma_{0}(q^{l-1}-1)-2 $ where $\sigma_{0}$ is the divisor count function.
   Also, It's obvious to see by Theorem ~\ref{carac3} that the sets $(\mathbb{Q}\setminus\mathbb{Z})\text{-}\mathcal{KS}(q^{l}) $ and $\mathbb{Q}\text{-}\mathcal{KS}(q^{l}) $ are infinite.

 Now, in the rest of this section, we will prove the existence of  connections between  some $q^{l}$-Korselt bases in $\mathbb{Q}$ by establishing relations between  $q^{l}$-Korselt bases in $\mathbb{Z}$  and others in $\mathbb{Q} \setminus \mathbb{Z}$. These relations  provide us to determine   the Korselt set of $q^{l}$ over $\mathbb{Q}$ from the knowledge of its  Korselt set over $\mathbb{Z}$ .

\begin{proposition}\label{conec}

Let   $N=q^{l}$ be such that $l\geq2$. Then, $\beta\in\mathbb{Z}$-$\mathcal{KS}(q^{l})$ if and only if  $\alpha=q+\dfrac{\beta-q}{s}\in(\mathbb{Q}\setminus\mathbb{Z})$-$\mathcal{KS}(q^{l})\setminus \{q+\dfrac{q^{l}-q}{s}\}$ for each $s\in\mathbb{N}\setminus\{0,1\}$ where $\gcd(s,\beta-q)=1$.

\end{proposition}

\begin{proof}
Let  $\beta\in\mathbb{Z}$, $s\in\mathbb{N}\setminus\{0,1\}$ with $\gcd(s,\beta-q)=1$.

Let $\alpha=q+\dfrac{\beta-q}{s}=\dfrac{\alpha_{1}}{\alpha_{2}}$ where $\alpha_{1}=sq+\beta-q$ and $\alpha_{2}=s$.

Since $\gcd(s,\beta-q)=1$, $\alpha_{2}q-\alpha_{1}=sq-(sq-(\beta-q))=\beta-q$ and
$$\alpha_{2}q^{l}-\alpha_{1}=sq^{l}-(sq-(\beta-q))=s(q^{l}-\beta)+(s+1)(\beta-q),$$

 the following equivalence holds

      $$ (q-\beta) \mid (q^{l}-\beta) \Leftrightarrow (\alpha_{2}q-\alpha_{1}) \mid (\alpha_{2}q^{l}-\alpha_{1}).$$

      As, in addition, $\beta\neq q^{l} \ \ \text{if and only if} \ \ \alpha\neq q+\dfrac{q^{l}-q}{s}$,
we conclude that
             $$ \beta\in\mathbb{Z}\text{-} \mathcal{KS}(q^{l})\ \ \text{if and only if} \ \ \alpha\in(\mathbb{Q}\setminus\mathbb{Z})\text{-}\mathcal{KS}(q^{l})\setminus \{q+\dfrac{q^{l}-q}{s}\}.$$

\end{proof}

By Proposition~\ref{conec},  we determine  in the next example, the sets $\mathbb{Z}\text{-}\mathcal{KS}(q^{l})$ and $\mathbb{Q}\text{-}\mathcal{KS}(q^{l})$ for some chosen numbers $q$ and $l$.
\medskip

\begin{example}\rm
\begin{enumerate}

\item Let $N_{2}=q^{2}$ where $q=2p+1$ and $p$ are  prime numbers.
 As $q^{2}-q=2pq$, we get

$$\begin{array}{lll}
  \mathbb{Z}\text{-}\mathcal{KS}(N_{2}) & = & \{ q\pm d \ \  ;  \ \  d\mid 2pq \}\setminus\{0,N_{2}\} \\
   & = & \{ q\pm d \ \  ; \ \ d\in\{1,2,p,q,2p,2q, pq,2pq\}   \}\setminus\{0,N_{2}\} \\
   & = & \{q\pm1,q\pm2,q\pm p,2q,q\pm2p,q\pm2q,q\pm pq,q-2pq\}.
\end{array}$$

Hence,
$$\begin{array}{ll}
  (\mathbb{Q}\setminus\mathbb{Z})\text{-}\mathcal{KS}(N_{2})  =\bigcup\limits_{s\in\mathbb{N}\setminus\{0,1\}} \{ q\pm\dfrac{d}{s};\ \ \gcd(d,s)=1 \,  \text{and} \ \  d\mid 2pq \}\setminus\{0,N_{2}\}& \\
    =\bigcup\limits_{s\in\mathbb{N}\setminus\{0,1\}}\{ q\pm\dfrac{d}{s};\ \  \gcd(d,s)=1, \, d\in\{1,2,p, q,2p,2q,pq,2pq\}\}\setminus\{0,N_{2}\}& \\
   = \bigcup\limits_{\substack{s\in\mathbb{N}\setminus\{0,1\}\\
                    \gcd(s,2pq)=1 }}
  \{q\pm\dfrac{1}{s},q\pm\dfrac{2}{s},q\pm \dfrac{p}{s},q\pm   \dfrac{q}{s},q\pm\dfrac{2p}{s},q\pm\dfrac{2q}{s},q\pm \dfrac{pq}{s},  q\pm\dfrac{2pq}{s}\}\setminus\{0,N_{2}\}&.

\end{array}$$

\bigskip

Thus, we obtain

$$ \begin{array}{ll}
\mathbb{Q}\text{-}\mathcal{KS}(N_{2})=\mathbb{Z}\text{-}\mathcal{KS}(N_{2})\cup(\mathbb{Q}\setminus\mathbb{Z})\text{-}\mathcal{KS}(N_{2})&\\
=\bigcup\limits_{\substack{s\in\mathbb{N}\setminus\{0\}\\
                    \gcd(s,2pq)=1 }}
\{q\pm\dfrac{1}{s},q\pm\dfrac{2}{s},q\pm \dfrac{p}{s},q\pm \dfrac{q}{s},q\pm\dfrac{2p}{s},q\pm\dfrac{2q}{s},q\pm \dfrac{pq}{s},q\pm\dfrac{2pq}{s}\}\setminus\{0,N_{2}\}.&
\end{array}$$
\vspace{3mm}

\item Let $N_{3}=q^{3}$ where $q=2p+1$, $p=2r-1$ and $r$ are  prime numbers. We denote the positive divisors of $q^{3}-q=8pqr$ by $D(q^{3}-q)$. Hence,
 $$ \begin{array}{lll}
 D(q^{3}-q) & = &\{d\in\mathbb{N}  \ \  \text{such that} \ \  d\mid 8pqr \}\\
    & = & \{1,2,4,8,p,q,r,2p,2q,2r,4p,4q,4r,8p,8q,8r,pq,pr,qr, \\
     &  & 2pq,2pr,2qr,4pq,4pr,4qr,8pq,8pr,8qr,pqr,2pqr,4pqr, 8pqr\}. \\

  \end{array}$$
 Therefore $$ \mathbb{Z}\text{-}\mathcal{KS}(N_{3})  = \{ q\pm d \ \  ; \ \ d\in D(q^{3}-q)  \}\setminus\{0,N_{3}\}$$
and $$(\mathbb{Q}\setminus\mathbb{Z})\text{-}\mathcal{KS}(N_{3})=\bigcup\limits_{s\in\mathbb{N}\setminus\{0,1\}}\{ q\pm\dfrac{d}{s}; \ \ \gcd(d,s)=1, \, d\in D(q^{3}-q)\}\setminus\{0,N_{3}\}.$$
Thus,
$$ \begin{array}{lll}
\mathbb{Q}\text{-}\mathcal{KS}(N_{3})&=&\mathbb{Z}\text{-}\mathcal{KS}(N_{3})\cup(\mathbb{Q}\setminus\mathbb{Z})\text{-}\mathcal{KS}(N_{3})\\
&=&\bigcup\limits_{s\in\mathbb{N}\setminus\{0\}}\{ q\pm\dfrac{d}{s}; \ \ \gcd(d,s)=1, \, d\in D(q^{3}-q)\}\setminus\{0,N_{3}\}.
\end{array}$$

\end{enumerate}
\end{example}

\bigskip
\section{Korselt weight of a $q^{l}$-rational base}

\begin{proposition}\label{bound2}

Let  $\alpha=\dfrac{\alpha_{1}}{\alpha_{2}}\in\mathbb{Q}\setminus\{0,1\}$ and $N=q^{l}$ be such that $l\geq2$ and $\gcd(q,\alpha_{1})=1$.
If $\alpha\in\mathbb{Q}$-$\mathcal{KS}(q^{l}) $,  then $2\leq q\leq \alpha+\abs{\alpha\alpha_{1}^{l-2}-\alpha_{2}^{l-2}} $.
\end{proposition}

\begin{proof}

 Let $\alpha=\dfrac{\alpha_{1}}{\alpha_{2}}\in\mathbb{Q}$-$\mathcal{KS}(N)$. Then, by Proposition ~\ref{carac0}$(1)$  we have $(\alpha_{2}q-\alpha_{1})\mid (q^{l-1}-1)$. So, we can write

$$\alpha_{2}q-\alpha_{1}\ \mid \alpha_{2}^{l-1}(q^{l-1}-1)=(\alpha_{2}q)^{l-1}-\alpha_{1}^{l-1}+\alpha_{1}^{l-1}-\alpha_{2}^{l-1},$$

hence $(\alpha_{2}q-\alpha_{1}) \mid (\alpha_{1}^{l-1}-\alpha_{2}^{l-1})$.  Since in addition $ \alpha\neq1$ (i.e $\alpha_{1}\neq\alpha_{2}$), it follows that
$\alpha_{2}q-\alpha_{1}\leq  \abs{\alpha_{1}^{l-1}-\alpha_{2}^{l-1}}$. Thus, $q\leq  \alpha+\abs{\alpha\alpha_{1}^{l-2}-\alpha_{2}^{l-2}}$.

\end{proof}
 By  Proposition~\ref{bound2} and for a fixed integer $l\geq2$, we can see  that  every nonzero rational number $\alpha\neq1$   can belong only to  a finite number of  sets $\mathbb{Q}$-$\mathcal{KS}(q^{l})$.  However, we will prove in the rest of this section that for each $\alpha\neq0$, there exist infinitely many couples $(l,q)$ ( equivalently \,$N=q^{l}$ ) such that $\alpha$ is an $N$-Korselt  base .

The following result shows  for $N=q^{l}$ that the exponent $l$ of $q$ can generate some $N$-Korselt rational bases.
\begin{proposition}\label{powgen}Let $N=q^{l}$ be such that $l\geq2$ and $\varphi(.)$ denote the Euler's totient function. Then  each  number $d$ where $\varphi(d)\mid(l-1)$, generates a Korselt base of $N$  in $\mathbb{Q}$.

\end{proposition}

\begin{proof}
 First, by  Euler's theorem, we have $d \mid(q^{\varphi(d)}-1)$. Furthermore, since $\varphi(d)\mid(l-1)$, we get $(q^{\varphi(d)+1}-q)\mid (q^{l}-q)$. Hence

\begin{equation}\label{eq2}
d  \mid (q^{l}-q).
  \end{equation}

By the euclidian division of $d$ by $q$ there exist $t_{d}\geq0$ and $0\leq r_{d}<q$ such that $d=t_{d}q+r_{d}$.
Three cases are to be discussed:

\begin{enumerate}
  \item Suppose that $t_{d}=0$ (i.e. $d<q$). Then by $\eqref{eq2}$ and for each $m\in\mathbb{N}\setminus\{0\}$, we get
   $mq-(mq-r_{d})=r_{d}=d\mid (q^{l}-q) $. Hence, $$\alpha_{d}=\dfrac{mq-r_{d}}{m}\in\mathbb{Q}\text{-}\mathcal{KS}(q^{l}).$$

  \item If $r_{d}=0$ (i.e. $q\mid d$), then there exist  $t_{d}^{'}$ and $t_{d}^{''}$  integers such that
  $t_{d}=t_{d}^{'}-t_{d}^{''}$ and $t_{d}^{'}t_{d}^{''}\neq0$. Consequently,  by $\eqref{eq2}$ we obtain
      $t_{d}^{'}-t_{d}^{''}=t_{d}\mid (q^{l-1}-1)$, this implies that $\alpha_{d}=\dfrac{t_{d}^{''}q}{t_{d}^{'}}\in\mathbb{Q}\text{-}\mathcal{KS}(q^{l}).$
 \item Now, assume that $r_{d}\neq0$ and $t_{d}\neq0$. Then  for each $m\in\mathbb{Z}$ where $t_{d}+m\neq 0$, we have
 $(t_{d}+m)q-(mq-r_{d})=d\mid (q^{l}-q)$. Thus, $$\alpha_{d}=\dfrac{mq-r_{d}}{t_{d}+m}\in\mathbb{Q}\text{-}\mathcal{KS}(q^{l}).$$
\end{enumerate}

\end{proof}
It follows by Proposition ~\ref{powgen}, that for each $N=q^{l}$ the set $\mathbb{Q}\text{-}\mathcal{KS}(q^{l})$ is non empty( $d=2$ is valid for all cases). Now, one may ask if  the converse is true.  The answer is yes, moreover  by the following result we show that for each  $\alpha\in\mathbb{Q}\setminus\{0\}$, there exist infinitely many $N=q^{l}$ such that $\alpha$ is an $N$-Korselt base.
\begin{theorem}\label{fermgen}
Let $\alpha=\dfrac{\alpha_{1}}{\alpha_{2}}\in\mathbb{Q}\setminus\{0\}$. Then the following assertions hold.

\begin{enumerate}
  \item There exist a prime number $p$ and $k\in\mathbb{N}$  such that $\gcd(\alpha_{1},p)=1$ and $\alpha\in\mathbb{Q}\text{-}\mathcal{KS}(p^{k})$.
  \item If $\alpha_{1}\neq 1$, then there exist a prime number $q$ and $l\in\mathbb{N}$  such that  $q$ divides $\alpha_{1}$ and $\alpha\in\mathbb{Q}\text{-}\mathcal{KS}(q^{l})$.
\end{enumerate}

\end{theorem}

\begin{proof}
 Let  $\alpha=\dfrac{\alpha_{1}}{\alpha_{2}}\in\mathbb{Q}\setminus\{0\}$.
 \begin{enumerate}
   \item Let  $p>\alpha_{1}$  be a prime number and  $k_{1}=\alpha_{2}p-\alpha_{1}\geq1$. Since $\gcd(k_{1},p)=1$, we get by  Euler's theorem
   $ \alpha_{2}p-\alpha_{1}=k_{1} \ \mid (p^{\varphi(k_{1})}-1)$.

  Hence $\alpha\in\mathbb{Q}\text{-}\mathcal{KS}(p^{k})$ where $k=\varphi(k_{1})+1$.

   \item Let $q$ be a prime number such that $\alpha_{1}=\alpha_{1}^{'}q$.

   By taking $l-1=\varphi(\mid\alpha_{2}-\alpha_{1}^{'}\mid)$, Euler's theorem implies that
 $$ (\alpha_{2}-\alpha_{1}^{'}) \mid (q^{\varphi(\mid\alpha_{2}-\alpha_{1}^{'}\mid)}-1)= q^{l-1}-1.$$

  Hence  $ (\alpha_{2}q-\alpha_{1})\mid (q^{l}- q)$, and so $\alpha\in\mathbb{Q}\text{-}\mathcal{KS}(q^{l})$.
 \end{enumerate}

\end{proof}

The next result follows immediately
\begin{corollary}\label{ferminfgen}
For each $\alpha=\dfrac{\alpha_{1}}{\alpha_{2}}\in \mathbb{Q}\setminus\{0\}$, the following assertions hold.
\begin{enumerate}
  \item There exist infinitely many $N=p^{k}$ with $p$ a prime number and $k$ a positive integer  such that $\gcd(\alpha_{1},p)=1$ and  $\alpha$ is a $K_{N}$-base.
  \item If $\alpha_{1}\neq 1$, then there exist infinitely many $N=q^{l}$ with $q$ a prime number and $l$ a positive integer such that $q$ divides $\alpha_{1}$ and  $\alpha$ is a $K_{N}$-base.
\end{enumerate}

\end{corollary}

\begin{proof}

\begin{enumerate}
  \item Follows directly from Theorem ~\ref{fermgen}$(1)$.
  \item Let $\alpha=\dfrac{\alpha_{1}^{'}q}{\alpha_{2}}\in \mathbb{Q}\setminus\{0\}$ and $l_{1}=\varphi(|\alpha_{2}-\alpha_{1}^{'}|)+1$. Then by Theorem~\ref{fermgen}.$(2)$, we have $\alpha\in\mathbb{Q}\text{-}\mathcal{KS}(q^{l_{1}})$. Hence,  for each  integer $l$ where  $(l_{1}-1)\mid (l-1)$, we get by Proposition ~\ref{inclusion1}$(2)$,
      $$\alpha\in\mathbb{Q}\text{-}\mathcal{KS}(q^{l_{1}})\subseteq\mathbb{Q}\text{-}\mathcal{KS}(q^{l}).$$  Therefore $\alpha$ is a $q^{l}$-Korselt base, and so $\alpha$ is an $N$-Korselt base for infinitely many numbers $N=q^{l}$.

\end{enumerate}

\end{proof}

By  Corollary ~\ref{ferminfgen}, we deduce that for each rational $\alpha\neq0$ the korselt weight of a base $\alpha$ over the set of prime powers is infinite.

\section{ Korselt set of $q^{l}$ over $[-1,1[\cap\mathbb{Q}$}

We begin this section by giving some  examples of  $q^{l}$-Korselt rational bases  in $[-1,1[$.
\begin{proposition}\label{infer1}
Let $q$ be a prime number and $l$ be a positive integer. Then $\dfrac{1}{q^{s-1}}\in\mathbb{Q}$-$\mathcal{KS}(q^{l})$ for each integer $s$ dividing $l-1$.
\end{proposition}

\begin{proof}

Since  $s\mid (l-1)$ that is  $s\mid (l+s-1)$, it follows that $q^{s-1}q-1=(q^{s}-1)\mid (q^{l+s-1}-1)=q^{s-1}N-1$. Thus, $\dfrac{1}{q^{s-1}}\in\mathbb{Q}\text{-}\mathcal{KS}(q^{l}).$

\end{proof}
Similarly, we obtain the following result
\begin{proposition}\label{infer-1}
Let $q$ be a prime number and $l$ be a positive  integer. Then $\dfrac{-1}{q^{s-1}}\in\mathbb{Q}$-$\mathcal{KS}(q^{l})$ for each positive integer $s$ dividing $l-1$ where $\dfrac{l-1}{s}$ is even.
\end{proposition}

\begin{proof}

Since $s\mid (l-1)$ hence $s\mid (l+s-1)$ and $\dfrac{l+s-1}{s}$ is odd, it follows that $q^{s-1}q+1=(q^{s}+1)\mid (q^{l+s-1}+1)=q^{s-1}N+1$. Therefore $\dfrac{-1}{q^{s-1}}\in\mathbb{Q}\text{-}\mathcal{KS}(q^{l}).$

\end{proof}

If $\dfrac{l-1}{s}$ is odd  the assertion in Proposition ~\ref{infer-1} is false, for instance,  if  $l=3$ and $q$ a  prime number, it's forward to verify that $\dfrac{-1}{q}\notin \mathbb{Q}$-$\mathcal{KS}(q^{3})$.

By the next two  results, we give some  relations between the  Korselt sets over $ \mathbb{Q}$ of the  prime powers $p^{l}$ and $q^{l}$.
\begin{proposition}\label{infer11}
Let $p,q$ be two distinct prime numbers and $l$ be a positive integer. Then, $\dfrac{1}{p}\in\mathbb{Q}$-$\mathcal{KS}(q^{l})$ if and only if \, $\dfrac{1}{q}\in\mathbb{Q}$-$\mathcal{KS}(p^{l})$.
\end{proposition}

\begin{proof} We have

$$\begin{array}{rll}
                 \dfrac{1}{p}\in\mathbb{Q}$-$\mathcal{KS}(q^{l})  &\Leftrightarrow & pq-1\,\mid \,pq^{l}-1 \\
                 & \Leftrightarrow &pq-1\,\mid \,p^{l-1}(pq^{l}-1)=(pq)^{l}-1-p^{l-1}+1  \\
                 & \Leftrightarrow &pq-1\,\mid \,p^{l-1}-1  \\
                 & \Leftrightarrow &qp-1\,\mid \,qp(p^{l-1}-1)+qp-1=qp^{l}-1\\
                  & \Leftrightarrow &\dfrac{1}{q}\in\mathbb{Q}$-$\mathcal{KS}(p^{l})
                 \end{array}.$$

\end{proof}

\begin{proposition}\label{equinfer-1}
Let $p,q$ be  two distinct prime numbers and  $l$ be a positive odd integer. Then the following assertion holds

$\dfrac{-1}{p}\in\mathbb{Q}$-$\mathcal{KS}(q^{l})$ if and only if \, $\dfrac{-1}{q}\in\mathbb{Q}$-$\mathcal{KS}(p^{l})$.
\end{proposition}
\begin{proof}
We have
 $$\begin{array}{rll}
                 \dfrac{-1}{p}\in\mathbb{Q}$-$\mathcal{KS}(q^{l})  &\Leftrightarrow & pq+1\,\mid \,pq^{l}+1 \\
                 & \Leftrightarrow &pq+1\,\mid \,p^{l-1}(pq^{l}+1)=(pq)^{l}+1+p^{l-1}-1  \\
                    & \Leftrightarrow & pq+1\,\mid \,p^{l-1}-1 \\
                   & \Leftrightarrow &pq+1 \,\mid \,pq(p^{l-1}-1)+pq+1=qp^{l}+1\\
                   & \Leftrightarrow &\dfrac{-1}{q}\in\mathbb{Q}$-$\mathcal{KS}(p^{l})
                 \end{array}.$$

 \end{proof}

If $l$ is even  the assertion in Proposition ~\ref{equinfer-1} is not true, for instance when $q=3$, $p=2$ and $l=4$, we have $ \dfrac{-1}{2}\in\mathbb{Q}$-$\mathcal{KS}(3^{4}) $ but $ \dfrac{-1}{3}\notin \mathbb{Q}$-$\mathcal{KS}(2^{4})$.

Now, by the  next two propositions and for a given prime number $q$  and  $l\in \mathbb{N}\setminus \{0,1\}$, we  determine  the set $([-1,1[\cap\mathbb{Q})$-$\mathcal{KS}(q^{l})$.

\begin{proposition}\label{set0}

Let $q$ be a prime number, $l\in\mathbb{N}\setminus\{0,1\}$ and

$\alpha=\dfrac{\alpha_{1}}{\alpha_{2}}\in \mathbb{Q}$-$\mathcal{KS}(q^{l})$. Then, $\alpha\in]0,1[$ if and only if there exists $d\mid(q^{l}-q)$ such that

  $$\alpha=q-\dfrac{d}{\alpha_{2}} \quad \text{and} \quad \alpha_{2}\in\left\{\Bigl\lfloor\dfrac{d}{q}\Bigr\rfloor+1,...,\Bigl\lceil\dfrac{d}{q-1}\Bigr\rceil-1\right\}.$$
\end{proposition}

\begin{proof}

By Theorem  ~\ref{carac3},  $\alpha=\dfrac{\alpha_{1}}{\alpha_{2}}\in ]0,1[\cap(\mathbb{Q}$-$\mathcal{KS}(q^{l}))$ is equivalent to the existence of a positive integer $d$   such that $d\mid(q^{l}-q)$  and $\alpha=q-\dfrac{d}{\alpha_{2}}$ . Further, we have

$$\begin{array}{lll}
                     \alpha\in]0,1[\cap (\mathbb{Q}$-$\mathcal{KS}(q^{l})) & \Leftrightarrow & 0<\alpha=q-\dfrac{d}{\alpha_{2}}<1 \\
                    & \Leftrightarrow & \dfrac{d}{q}<\alpha_{2}<\dfrac{d}{q-1} \\
                    & \Leftrightarrow &  \Bigl\lfloor\dfrac{d}{q}\Bigr\rfloor+1\leq\alpha_{2}\leq\Bigl\lceil\dfrac{d}{q-1}\Bigr\rceil-1 \\
                   & \Leftrightarrow & \alpha_{2}\in\left\{\Bigl\lfloor\dfrac{d}{q}\Bigr\rfloor+1,...,\Bigl\lceil\dfrac{d}{q-1}\Bigr\rceil-1\right\},
                 \end{array}$$

which complete the proof.
\end{proof}

\begin{corollary}\label{emptset1} Let $q$ be a prime number and $\mathbb{A}=]0,1[\cap\mathbb{Q}$. Then $$\mathbb{A}\text{-}\mathcal{KS}(q^{2})=\emptyset.$$

\end{corollary}

\begin{proof}
 Suppose   that there exists $\alpha\in]0,1[\cap (\mathbb{Q}$-$\mathcal{KS}(q^{2}))$. This is equivalent by Theorem  ~\ref{carac3} and Proposition ~\ref{set0}, to the exitence of  $d\mid (q^{2}-q)$ such that $\alpha=q-\dfrac{d}{\alpha_{2}}$ and
\begin{equation}\label{eq3}
\Bigl\lfloor\dfrac{d}{q}\Bigr\rfloor+1\leq\alpha_{2}\leq\Bigl\lceil\dfrac{d}{q-1}\Bigr\rceil-1.
  \end{equation}
Since, in addition $d\mid(q^{2}-q)=q(q-1)$, two cases are to be considered:

   \begin{itemize}
      \item If $\gcd(q,d)=1$, then $d$ divides $q-1$ and so $d<q$. Hence, by $\eqref{eq3}$ $$1=\Bigl\lfloor\dfrac{d}{q}\Bigr\rfloor+1\leq\Bigl\lceil\dfrac{d}{q-1}\Bigr\rceil-1=0$$ which is impossible.
      \item Assume that $d=d^{'}q$. Then, since $d=d^{'}q\mid q(q-1)$ hence $d^{'}\mid(q-1)$, it follows by $\eqref{eq3}$ that
      $$d^{'}+1=\Bigl\lfloor\dfrac{d}{q}\Bigr\rfloor+1\leq\Bigl\lceil\dfrac{d}{q-1}\Bigr\rceil-1=d^{'}$$
           which is again impossible.
    \end{itemize}
\end{proof}

\begin{proposition}\label{set1}

Let $q$ be a prime number, $l\in\mathbb{N}\setminus\{0,1\}$ and

$\alpha=\dfrac{\alpha_{1}}{\alpha_{2}}\in \mathbb{Q}$-$\mathcal{KS}(q^{l})$. Then,  $\alpha\in[-1,0[$ if and only if there exists $d \mid(q^{l}-q)$ such that

 $$\alpha=q-\dfrac{d}{\alpha_{2}} \quad \text{and} \quad \alpha_{2}\in\left\{\Bigl\lceil\dfrac{d}{q+1}\Bigr\rceil,...,\Bigl\lceil\dfrac{d}{q}\Bigr\rceil-1\right\}.$$

\end{proposition}

\begin{proof}
By Theorem  ~\ref{carac3},  $\alpha=\dfrac{\alpha_{1}}{\alpha_{2}}\in [-1,0[\cap(\mathbb{Q}$-$\mathcal{KS}(q^{l}))$ is equivalent to the existence of an integer $d>0$   such that $d\mid(q^{l}-q)$  and $\alpha=q-\dfrac{d}{\alpha_{2}}$ . Also, the following equivalences hold

$$\begin{array}{lll}
                     \alpha\in[-1,0[\cap (\mathbb{Q}$-$\mathcal{KS}(q^{l})) & \Leftrightarrow & -1\leq \alpha=q-\dfrac{d}{\alpha_{2}}<0  \\
                    & \Leftrightarrow & \dfrac{d}{q+1}\leq\alpha_{2}<\dfrac{d}{q} \\
                    & \Leftrightarrow &  \Bigl\lceil\dfrac{d}{q+1}\Bigr\rceil\leq\alpha_{2}\leq\Bigl\lceil\dfrac{d}{q}\Bigr\rceil-1 \\
                   & \Leftrightarrow & \alpha_{2}\in\left\{\Bigl\lceil\dfrac{d}{q+1}\Bigr\rceil,...,\Bigl\lceil\dfrac{d}{q}\Bigr\rceil-1\right\}.
                 \end{array}$$
Thus the proposition is proved.

\end{proof}

\begin{corollary}\label{emptset2} Let $q$ be a prime number and $\mathbb{A}=[-1,0[\cap\mathbb{Q}$. Then

$$\mathbb{A}\text{-}\mathcal{KS}(q^{2})=\emptyset.$$

\end{corollary}

\begin{proof}
Suppose by contradiction that there exists $\alpha\in[-1,0[\cap (\mathbb{Q}$-$\mathcal{KS}(q^{2}))$.
  This is equivalent by Proposition~\ref{set1}, to the existence of a positive integer $d\mid(q^{2}-q)$ such that $\alpha=q-\dfrac{d}{\alpha_{2}}$ and

\begin{equation}\label{eq4}
\Bigl\lceil\dfrac{d}{q+1}\Bigr\rceil\leq\alpha_{2}\leq\Bigl\lceil\dfrac{d}{q}\Bigr\rceil-1.
  \end{equation}
Since $d\mid(q^{2}-q)=q(q-1)$, two cases are to be discussed:

   \begin{itemize}
      \item If $\gcd(q,d)=1$, then $d$ divides $q-1$ and so $d<q$. Hence, by $\eqref{eq4}$ $$1=\Bigl\lceil\dfrac{d}{q+1}\Bigr\rceil\leq\Bigl\lceil\dfrac{d}{q}\Bigr\rceil-1=0$$ which is impossible.
      \item If $q\mid d$, then $d=d^{'}q$ with $d^{'}\in\mathbb{N}$. Since $d=d^{'}q\mid q(q-1)$ hence $d^{'}\mid(q-1)$, it follows by $\eqref{eq4}$, that
      $$d^{'}+1=\Bigl\lfloor\dfrac{d}{q}\Bigr\rfloor+1\leq\Bigl\lceil\dfrac{d}{q-1}\Bigr\rceil-1=d^{'}$$
           which is again impossible.
    \end{itemize}

    So, we conclude that  $\mathbb{A}\text{-}\mathcal{KS}(q^{2})$ is empty.
\end{proof}

\end{document}